\documentclass[11pt]{article}

\usepackage[top=1in, left=1in, bottom=1.5in, right=1.5in]{geometry}

\usepackage{hyperref}  %provides clickable links to references within the document
\usepackage{setspace} 

\usepackage{graphicx}
\usepackage{amsfonts}
\usepackage{amssymb}
\usepackage{amsmath}
\usepackage{amsthm}

\newtheorem{theorem}{Theorem}[section]

\newtheorem{lemma}[theorem]{Lemma}
\newtheorem{conjecture}[theorem]{Conjecture}
\newtheorem{claim}{}[theorem]

\newcommand{\del}{\backslash}

\newcommand{\cB}{\mathcal{B}}

\title{A counterexample to Hickingbotham's conjecture\\
 about $k$-ghost-edges}

\author{Rong Chen\\
\\
Center for Discrete Mathematics,\ \ Fuzhou University\\
Fuzhou,\ \ P. R. China}
\begin{document}

\maketitle

\footnote{

Email: rongchen@fzu.edu.cn
}

\begin{abstract}
Fix $k\in \mathbb{N}$ and let $G$ be a connected graph with $tw(G)\leq k$. We say that $xy\in E(G^c)$ is a {\em $k$-ghost-edge} of $G$ if for every tree decomposition $(T,\cB)$ of $G$ with width at most $k$, the set $\{x,y\}$ is contained in a bag of $(T,\cB)$. Although a $k$-ghost-edge of $G$ is not an edge of $G$, but it behaves like real edges with respect to tree decomposition of $G$ with width at most $k$. For any graph $G$ with treewidth $k$ and $xy\in E(G^c)$, when there are at least $k+1$ internally vertex disjoint $(x,y)$-paths, Hickingbotham proved that $xy$ is a $k$-ghost-edge of $G$; while when there are at most $k$ internally vertex disjoint $(x,y)$-paths, he conjectured that it is not a $k$-ghost-edge of $G$. In this paper, we prove that this conjecture is wrong.

\end{abstract}

\section{Introduction}

All graphs considered in this paper are finite and simple. 
%Treewidth, pathwidth and treedepth paly a major role in the structural graph theory, in particular in the study of general sparse graph classes. Every $\ell$-vertex path has treedepth $\text{log}_2{\ell}$, and every graph with no such path has treedepth less than $\ell$. Similarly, pathwidth is approximated by maximum height of a complete binary tree minor: every graph containing a complete binary tree of height $h$ as a minor has pathwidth at least $\lfloor{h \over 2}\rfloor$, and every graph with no such minor has pathwidth at most $\mathcal{O}(2^h)$. For bath parameters, the exponential gap between the respective lower and upper bounds can not be avoided, as witnessed by complete graphs. Treewidth is approximated by the maximum size of a grid minor, but the gap is polynomial: while every graph containing a $k\times k$ grid as a minor has treewidth at least $k$, every graph with no such minor has treewith polynomial in $k$.
Tree decompositions and path decompositions are fundamental objects in graph theory. If a graph $G$ has large pathwidth and there is a tree decomposition $(T,\cB)$ of $G$ with small width and such that $T$ is a subtree of $G$, then $T$ will also have large pathwidth \cite{Hick19}. The crucial question is when does a graph $G$ have a tree decomposition with small width that is indexed by a subtree of $G$. More generally speaking, suppose that a graph $G$ has small treewidth, and consider all tree decompositions $(T,\cB)$ whose width is not too much larger than the optimum. To what extent can we choose or manipulate the ``shape'' of $T$?

For graphs with no long path, we can choose $T$ to also have no long path; this gives rise to
the parameter called treedepth \cite{NdP12}. Similarly, for graphs of bounded degree, we can choose $T$ to also have bounded degree \cite{DO95}; this relates to the parameters of congestion and dilation. Moreover, for graphs excluding any tree as a minor, we can choose $T$ to just be a path; this results in the parameter called pathwidth \cite{BRST91}. It would be wonderful if we could unify all such results into a single theorem which relates the shape of $T$ to $G$. In 2019, Dvo\v r\'ak suggested one way of accomplishing this goal. In the conjecture below and throughout this paper, we write $tw(G)$ for the treewidth of $G$.

\begin{conjecture}\label{ques}(\cite{D19})
There exists a polynomial function $f$ such that every connected graph $G$ has a tree decomposition $(T,\cB)$ of width at most $f(tw(G))$ such that $T$ is a subgraph of $G$.
\end{conjecture}
Hickingbotham is the first one to study Conjecture \ref{ques}. Although he found a family of graphs satisfing Conjecture \ref{ques}, but he conjectured it wrong \cite{Hick19}.

\begin{conjecture}\label{conj-1}(\cite{Hick19})
For any integer $k\geq1$, there is a connected graph $G_k$ with $tw(G_k)\leq2$  such that every tree decomposition $(T,\cB)$ of $G_k$ has width at least $k$, where $T$ is a subtree of $G$.
\end{conjecture}

Blanco et.al in \cite{BCHHIM23} independently considered Conjecture \ref{ques} and proved that Conjecture \ref{conj-1} is true. Ghost edges is a tool defined by Hickingbotham in \cite{Hick19} to construct graphs satisfying Conjecture \ref{conj-1}. %Examples given in \cite{BCHHIM23} to prove that Conjecture \ref{conj-1} is true have many ghost edges.
Fix $k\in \mathbb{N}$ and let $G$ be a connected graph with $tw(G)\leq k$. We say that $xy\in E(G^c)$ is a {\em $k$-ghost-edge} of $G$ if for every tree decomposition $(T,\cB)$ of $G$ with width at most $k$, we  have $x,y\in B_t$ for some $t\in V(T)$. Although $k$-ghost-edges of $G$ are not in $E(G)$, but they behave like real edges with respect to tree decomposition with width at most $k$. Hickingbotham in \cite{Hick19} proved that for any graph $G$ with treewidth $k$ and $xy\in E(G^c)$, if there are at least $k+1$ internally vertex disjoint $(x,y)$-paths, then $xy$ is a $k$-ghost-edge of $G$; %and conjectured the following conjecture and Conjecture \ref{ques} holding for graphs without $k$-ghost-edges. %
and proposed the following conjecture.

\begin{conjecture}\label{conj}(\cite{Hick19})
For any graph $G$ with $tw(G)\leq k$ and $xy\in E(G^c)$, if there are at most $k$ internally vertex disjoint $(x,y)$-paths, then $xy$ is not a $k$-ghost-edge of $G$.
\end{conjecture}

In this paper, we prove that %give a counterexample to Conjecture \ref{conj}. That is,

\begin{theorem} \label{main result}
Conjecture \ref{conj} is not true.
\end{theorem}

%%%%%%%%%%%%%%%%%%%%%%%%%%%%%%%%%%%%%%%%%%%%%%%%%%%%%%%%%%%%%%%%%%%%%%%%%%%%%%%%%%%%%%%%%%%%%%%%%%%%%%%%%%%
%%%%%%%%%%%%%%%%%%%%%%%%%%%%%%%%%%%%%%%%%%%%%%%%%%%%%%%%%%%%%%%%%%%%%%%%%%%%%%%%%%%%%%%%%%%%%%%%%%%%%%%%%%%
\section{Proof of Theorem \ref{main result}} %Finding a large complete bipartite minor }

Let $G$ be a graph and $T$ be a tree. Let $\cB:=\{B_t \subseteq V(G): t \in V(T)\}$ be a set of subsets of $V(G)$ indexed by the vertices of $T$. Each $B_x$ is called a \emph{bag}. Set $T_v:=T[\{t \in V(T): v \in B_t\}]$. The pair $(T, \cB)$ is a \emph{tree-decomposition} of $G$ if
\begin{enumerate}
    \item[(T1)] $V(G)=\bigcup_{t \in V(T)} B_t$;
    \item[(T2)] for each edge $ uv \in E(G)$, there is a vertex $t \in V(T)$ such that $u,v \in B_t$;
    \item[(T3)] for each vertex $v\in V(G)$, the subgraph $T_v$ of $T$ is connected.
\end{enumerate}
The {\em width} of  $(T, \cB)$ is then the maximum, over all $t\in V(T)$, of $|B_t|-1$. The {\em treewidth} of  $G$, denoted by $tw(G)$, is the minimum width of a tree decomposition of $G$. If a tree decomposition  $(T, \cB)$  of $G$ has width $tw(G)$, we say that $(T, \cB)$ is {\em optimal}.

For any subgraph $H$ of $G$, similarly we can define $T_H:=T[\{t \in V(T): B_t\cap V(H)\neq\emptyset\}]$ and $\cB\cap V(H):=\{B_t\cap V(H):\  t \in V(T)\}$. Lemma \ref{subgraph} is a folklore result and follows immediately from the definition of tree decomposition.
\begin{lemma}\label{subgraph}
Let $(T, \cB)$ be a tree decomposition of a graph $G$. For every connected subgraph $H$ of $G$, the graph $T_H$ is connected and $(T_H,\cB\cap V(H))$ is also a tree decomposition of $H$.
\end{lemma}

\begin{figure}[htbp]
\begin{center}
\includegraphics[height=7cm,page=1]{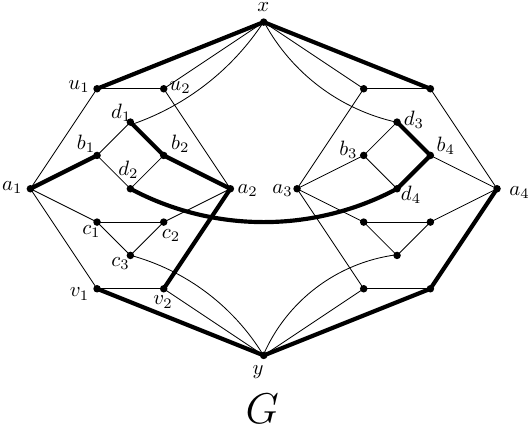}
\caption{The bold connected subgraphs in $G$ are branch sets of a $K_5$-minor of $G$. }
\label{Gn}
\end{center}
\end{figure}

\begin{figure}[htbp]
\begin{center}
\includegraphics[height=5cm,page=2]{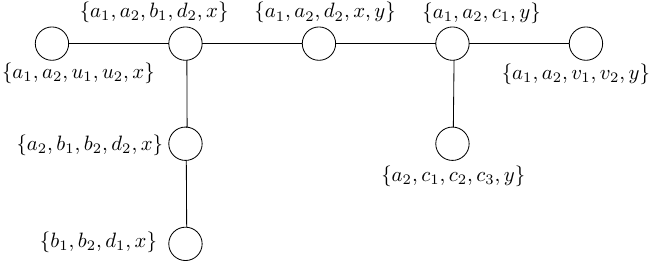}
\caption{a tree decomposition of $H_1$ with width $4$ and with $\{x,y,d_2\}$ contained in a bag. }
\label{tree}
\end{center}
\end{figure}

\begin{theorem}\label{main}
Let $G$ be the graph pictured as Figure 1. Then $tw(G)=4$ and $xy$ is a $4$-ghost-edge of $G$.
\end{theorem}
\begin{proof}
First we show that $tw(G)=4$. Since $G$ contains a minor isomorphic to $K_{5}$ (the bold connected subgraphs in Figure 1 are branch sets of a $K_5$-minor of $G$), we have $tw(G)\geq 4$. Hence, it suffices to show that $tw(G)\leq 4$.
Let $H_1, H_2$ be the subgraphs of $G$ induced by the union of $\{x,y\}$ and the vertex set of the component of $G\del \{x,y,d_2d_4\}$ containing $a_1$ and $a_3$, respectively. Assume that there is a tree-decomposition $(T_1,\cB_1)$ of $H_1$ with width 4 such that $\{x,y,d_2\}$ is contained in a bag $B_{s_1}$ of $(T_1,\cB_1)$ for some $s_1\in V(T_1)$. Since $H_1,H_2$ are isomorphic graphs, there is also a tree-decomposition $(T_2,\cB_2)$ of $H_2$ with width 4 such that $\{x,y,d_4\}$ is contained in a bag $B_{s_2}$ of $(T_2,\cB_2)$ for some $s_2\in V(T_2)$. Let $T$ be a tree obtained from $T_1\cup T_2$ by adding a new vertex $s$ linking $s_1$ and $s_2$. For any $u\in V(T)$, define $B_u$ the same as its corresponding bag in $(T_i,\cB_i)$ when $u\in V(T_i)$ for some $1\leq i\leq2$, and set $B_s:=\{x,y,d_2,d_4\}$. Since $\{x,y\}=V(H_1)\cap V(H_2)$, $\{x,y,d_2\}\subseteq B_{s_1}$ and $\{x,y,d_4\}\subseteq B_{s_2}$, for any vertex $v\in V(G)$, the subgraph $T_v$ is connected. Moreover, since $B_s=\{x,y,d_2,d_4\}$ and $d_2d_4$ is the only edge of $G$ not in $H_1$ or $H_2$, the defined $(T,\cB)$ is a tree decomposition of $G$ with width 4. Hence, it suffices to show that Figure 2 is a tree decomposition of $H_1$, which is easy to see that it is true. Hence, $tw(G)\leq 4$, implying $tw(G)=4$.

Next we show that $xy$ is a $4$-ghost-edge of $G$. Set $A:=\{a_1,a_2,a_3,a_4\}$.
\begin{claim}\label{S}
Let $S$ be a subset of $V(G)$ with $|S|\leq4$ and $S\cap\{x,y\}=\emptyset$. If there is no $(x,y)$-path in $G\del S$, then $S=A$.
\end{claim}
\begin{proof}
Since there are four internally vertex disjoint $(x,y)$-paths in $G$ and $|S|\leq4$, we have $|S|=4$. Moreover, since the components $C_1,C_2$ of $G\del\{x,y,d_2d_4\}$ are 2-connected, $|S\cap V(C_1)|=|S\cap V(C_2)|=2$. By symmetry we may assume that $a_1\in V(C_1)$. Note that except $\{a_1,a_2\}$, no two-vertex set of $C_1$ intersects all $(\{u_1,u_2,d_1\},\{v_1,u_2,c_3\})$-paths of $C_1$, so $S\cap V(C_1)=\{a_1,a_2\}$, where an $(s,t)$-path path $P$ of $G$ is an {\em $(S,T)$-path} if $s\in S, t\in T$ and  %one end of $P$ is in $X$ and the other end is in $Y$, and
no interior vertex of $P$ is contained in $S\cup T$ for any disjoint subsets $S,T$ of $V(G)$. By the symmetry between $C_1$ and $C_2$, we have $S\cap V(C_2)=\{a_3,a_4\}$, so $S=A$.
\end{proof}

Assume that $xy$ is not a 4-ghost-edge of $G$. Then there is an optimal tree decomposition $(T,\cB)$ of $G$ such that $V(T_x)\cap V(T_y)=\emptyset$ as $tw(G)=4$. Let $s\in V(T_x)$ and $t\in V(T_y)$ be the vertices that realise the closest distance between $T_x$ and $T_y$. Without loss of generality we may further assume that no bag of $(T,\cB)$ is a subset of another bag.
\begin{claim}\label{}
$B_s=A\cup\{x\}$ and $B_t=A\cup\{y\}$.
\end{claim}
\begin{proof}
By the definition of $s,t$ and Lemma \ref{subgraph}, there is no $(x,y)$-path in $G\del(B_s\del\{x\})$ and $G\del(B_t\del\{y\})$, so $B_s\del\{x\}=B_t\del\{y\}=A$ by Claim \ref{S} as $|B_s|,|B_t|\leq5$. This proves this claim.
\end{proof}
Set $H:=G[B_s\cup\{b_1,b_2,b_3,b_4,d_1,d_2,d_3,d_4\}]$. Since $H$ is connected, $T_H$ is connected and $(T_H,\cB\cap V(H))$ is a tree-decomposition of $H$ by Lemma \ref{subgraph}. When $s$ is a leaf vertex of $T_H$, since each vertex in $B_s$ has a neighbour in $V(H)-B_s$, we have $B_{s'}\cap V(H)=B_s$ by (T2), implying $B_{s'}=B_s$ as $(T,\cB)$ has width 4, where $s'$ is the unique neighbour of $s$ in $T_H$, which is a contradiction as no bag of $(T,\cB)$ is a subset of another bag. So $s$ is not a leaf vertex of $T_H$, implying that $H\del B_s$ is not connected, a contradiction to the fact that $H\del B_s$ is connected.
\end{proof}

\begin{proof}[Proof of Theorem \ref{main result}.]
Let $G$ be the graph pictured as Figure 1. Since there are four internally vertex disjoint $(x,y)$-paths in $G$, Conjecture \ref{conj} is not true by Theorem \ref{main}.
\end{proof}

\section{Acknowledgments}
This research was partially supported by grants from the Natural Sciences Foundation of Fujian Province (grant 2025J01486).
%and the Open Project Program of Key Laboratory of Discrete Mathematics with Applications of Ministry of Education (grant J20250601), Fuzhou University. The author thanks Zijian Deng for carefully reading this paper and finding an err in the proof of Lemma \ref{prism}. %grants from the National Natural Sciences Foundation of China (No. 11971111).

\end{document}